\documentclass[preprint,12pt]{elsarticle}

\usepackage{amsmath,amssymb,amsbsy,amstext,amscd,amsfonts}
\usepackage{graphicx, color}
\usepackage{geometry}
\usepackage{subfigure}

\newcommand{\bC}{\mathbb{C}}
\newcommand{\bZ}{\mathbb{Z}}
\newcommand{\bN}{\mathbb{N}}

\newcommand{\cV}{\mathcal{V}}
\newcommand{\cG}{\mathcal{G}}

\newcommand{\cH}{\mathcal{H}}

\newcommand{\cK}{\mathcal{K}}

\newcommand{\Rint}{\mathring{R}}

\newtheorem{theorem}{\noindent {\rm \bf Theorem}}[section]

\newtheorem{lemma}[theorem]{\noindent {\rm \bf Lemma}}

\newenvironment{proof}{\begin{trivlist}\item[]{\sl Proof.\/\ }}
                      {\hfill $\Box$
                      \end{trivlist}}

\journal{arXiv.org}

\begin{document}
\begin{frontmatter}

\title{Exterior diffraction problems for a triangular lattice}

\author[rmi,fre]{D. Kapanadze \corref{cor1}}
\ead{david.kapanadze@gmail.com}

\author[rmi]{E. Pesetskaya }
\ead{kate.pesetskaya@gmail.com}

\cortext[cor1]{Corresponding author}
\address[rmi]{A. Razmadze Mathematical Institute, TSU, Merab Aleksidze II Lane 2, Tbilisi 0193, Georgia}
\address[fre]{Free University of Tbilisi, Tbilisi 0159, Georgia}

\begin{abstract}
Exterior Dirichlet problems for two-dimensional lattice waves on the semi-infinite triangular lattice are considered. Namely, we study Dirichlet problems for the two-dimensional discrete Helmholtz equation in a plane with a hole. New results are obtained for the existence and uniqueness of the solution in the case of the real wave number $k\in (0,2\sqrt{2})$ without passing to a complex wave number. Besides, Green's representation formula for the solution is derived with the help of difference potentials. To demonstrate the results, we propose a method for numerical calculation.
\end{abstract}


%\textbf{Mathematics Subject Classification (2010)}. 78A45, 35J05, 39A14, 39A60, 74S20
\end{frontmatter}

\section{Introduction}

Nowadays there is an increased industrial and scientific interest in the study of nano- and microstructures of modern materials and composites.
Consideration of discrete structures of the materials is one of the ways to investigate microstructural processes in them, cf., e.g., \cite{Br, CI, Do, Sl}. Therefore, we devote our paper to the study of the exterior Dirichlet problem for the discrete Helmholtz equation.

Continuum models with sufficiently smooth boundaries are well studied (cf. \cite{KC} and references wherein).
However, the derivation of a discrete analogue of the Rayleigh-Sommerfeld scattering theory for different types of lattices is still under development and has many applications. There are five two-dimensional Bravais lattice types which naturally appear in application, and a triangular lattice is one of them. Some structures of left-handed 2D metamaterials \cite{CI} (which are a host microstrip line network periodically loaded with series capacitors and shunt inductors for signal processing and filtering), close-packed planes in some kinds of crystals \cite{BH54, Bu66} can be represented by the triangular lattices. Therefore, in this paper we analyze the exterior Dirichlet problem for the discrete Helmholtz equation in the triangular lattice mathematically formulated in Section \ref{sec:2}. Although a similar problem for a square lattice has been studied in \cite{DK}, its extension to the triangular lattice model is not direct.

In this paper, we obtain new results on the existence and uniqueness of a solution in the case of a real wave number $k\in (0,2\sqrt{2})$. It is well known that for the negative discrete Laplacian in the triangular lattice the spectrum is (absolutely continuous) $[0, 9]$, but there is an exceptional set $\{0, 8, 9\}$ in $[0, 9]$ where the limiting absorption principle fails \cite{AIM16, DK1}. Therefore, for “admissible” wave numbers, one can study the
problem as the limit $k+\iota 0$ of the complex wave number. This method is applied by Sharma in \cite{Sh}, where diffraction on triangular and hexagonal lattices by a finite crack and rigid constraint is investigated. In the present paper, we carry out our investigation without passing to the complex wave number. For this purpose, we use the radiation conditions and asymptotic estimates from \cite{DK1} described in Section \ref{sec:3} and Rellich-Vekua type theorem from Isozaki el al \cite{AIM16}. In Section \ref{sec:4}, we prove the unique solvability results and obtain a representation formula for the solution to the problem under consideration. Notice that our results will help us to analyze the diffraction of lattice waves by various types of defects including straight and zigzag rigid constraints and cracks. To demonstrate purposes, we take a small defect and present some numerical results in Section \ref{sec:5}. For the numerical calculation, we apply the method developed in \cite{BC} which allows us to calculate the lattice Green's functions without the need to perform integrals and appears to be much more effective. Finally, it is worth mentioning that due to the more complex form of the radiation condition for $k\in (2\sqrt{2}, 3)$, cf. \cite{DK1}, we have some difficulties to prove the uniqueness result for this case using the proposed method and, therefore, we consider only the case $k\in (0,2\sqrt{2})$.

\section{Formulation of the problem}
\label{sec:2}

Let us consider $\{\mathcal{V}, \mathcal{E}\}$ a periodic simple graph defining a two-dimensional infinite triangular lattice $\mathfrak{T}$, where
\begin{equation}
\label{eq:Vmain}
\mathcal{V}=\{T(x_1,x_2)\subset \mathbb{R}^2 : (x_1,x_2)\in \bZ^2=\bZ\times \bZ\}
\end{equation}
is a vertex set, $\mathcal{E}$ is an edge set, whose endpoints $(v,w)\in \mathcal{V}\times \mathcal{V}$ are adjacent points, i.e., $| v-w|=1$, and $T$ denotes a 2-dimensional coordinate transformation defined as
\[
T(x_1,x_2) = (x_1+x_2/2,\sqrt{3}x_2/2).
\]

We define for any point $v\in\mathcal{V}$ the 6-neighbourhood $F^0_v$ as the set of points $w\in \mathcal{V}$ such that $| v-w|=1$ and the neighbourhood $F_v$ as $F^0_v\bigcup\{v\}$. Recall that $R\subset \mathcal{V}$ is a region if there exist disjoint nonempty subsets $\Rint$ and $\partial R$ of $R$ such that
\begin{itemize}
\item[(a)] $R=\Rint\cup \partial R$,
\item[(b)] if $v\in \mathring{R}$ then $F_v\subset R$,
\item[(c)] if $v\in\partial R$ then there is at least one point $w\in F^0_v$ such that $w\in \Rint$.
\end{itemize}

As since the subsets $\Rint$ and $\partial R$ are not defined uniquely by $R$, henceforth, it is assumed for a given region $R$ in $\mathcal{V}$
that $\Rint$ and $\partial R$ are given and fixed. We say that $v$ is an interior (boundary) point of $R$ if $v\in \Rint$ ($v\in \partial R$). Further, a region $R\subset\mathcal{V}$ is said to be connected if for any $w,\tilde{w}\in R$ there exists a sequence $v^{(1)},\dots, v^{(n)}\in R$ with $v^{(1)}=w$  and $v^{(n)}=\tilde{w}$ such that for all $0\le i\le n-1$, $| v^{(i)}-v^{(i+1)} | =1$. By definition, a region $R$ with one interior point $v$ is connected and coincides with $F_v$.

\begin{figure}[htbp]
\centering
\includegraphics[scale=0.7]{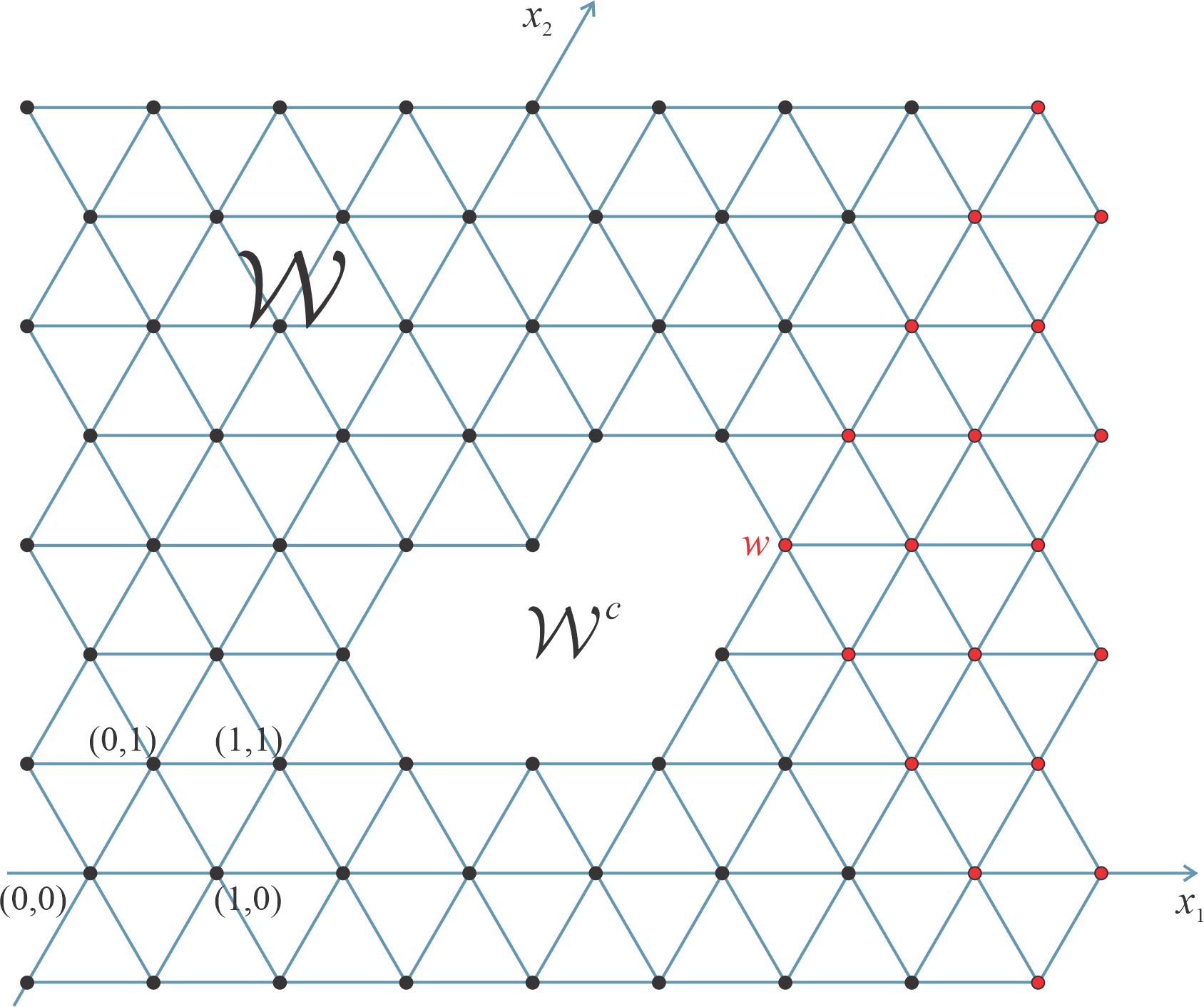}
\caption{Exterior problem in $\mathcal{V}$. Connection between $x=(x_1,x_2)\in \bZ^2$ and the Euclidean coordinates of the vertexes is established via  $(x_1,x_2) \to (x_1+x_2/2,\sqrt{3}x_2/2)$. For the boundary point $w$, the subset of the ``cone'' $C_{0}(w)$ is represented by the red dots.}
\label{Fig1}
\end{figure}

Let $\mathcal{W}$ be a connected region in $\mathcal{V}$ such that its complement $\mathcal{W}^c=\mathcal{V}\backslash \mathcal{W}$ is an empty set or a set with a finite number of lattice points. If $\mathcal{W}^c=\varnothing$ then we have the case $\mathring{\mathcal{W}}\cup\partial\mathcal{W}=\mathcal{V}$. When $\mathcal{W}^c\neq\varnothing$ then we additionally require that
$\mathcal{W}^c$ is a finite region such that $\partial\mathcal{W}^c=\partial\mathcal{W}$. Consequently, we have a disjoint decomposition
\[
\mathring{\mathcal{W}}\cup\partial\mathcal{W}\cup \mathring{\mathcal{W}}^c=\mathcal{V}.
\]

To guarantee the uniqueness of the solution, we suppose that $\mathcal{W}$ satisfies the cone condition. Namely, for any $w=(w_1,w_2)\in\mathcal{W}$, there is a ``cone'' $C_{i}(w)$, $i=0,1,\dots,5,$ such that $C_{i}(w)\subset\mathcal{W}$. Here, $C_{i}(w)$ is defined as follows
\begin{equation*}
\begin{aligned}
C_{i}(w)=\big\{v=(v_1,v_2)\in\mathcal{V}:& | -\sin(\pi i/3)(v_1-w_1)+\cos(\pi i/3)(v_2-w_2) | \\
&\le \sqrt{3} (\cos(\pi i/3) (v_1-w_1) +\sin(\pi i/3)(v_2-w_2))\big\}.
\end{aligned}
\end{equation*}
In particular, when $i=0$, we have
\[
C_{0}(w)=\big\{v=(v_1,v_2)\in\mathcal{V}: | (v_2-w_2) | \le \sqrt{3} (v_1-w_1)\big\}.
\]

The time-harmonic discrete waves in $\mathcal{W}=\mathring{\mathcal{W}} \cup \partial \mathcal{W}$ can be described by solutions of the following discrete Helmholtz equation
\begin{equation}\label{eq:Helmholtz2}
(\Delta_d + k^2)U(w)=0, \quad w=(w_1,w_2)\in \mathcal{W},
\end{equation}
where $\Delta_d$ denotes the discrete (a 7-point) Laplacian
\begin{equation}
\begin{aligned}
\Delta_d U(w)=\sum_{v\in F^0_w}U(v)-6U(w).
\end{aligned}
\end{equation}

We state the problem to find a unique solution $U$ to the discrete Helmholtz equation in $\Omega$ satisfying the non-homogeneous Dirichlet problem:
\begin{subequations}
\begin{align}
(\Delta_d +k^2)U(w) &= 0,\quad\quad\  \ \textup{in}\  \mathring{\mathcal{W}}, \label{eq:H1} \\
U(w)&= f(w),\quad \textup{on}\  \partial \mathcal{W}. \label{eq:H2}
\end{align}
\end{subequations}
Here, $f: \partial \mathcal{W} \rightarrow \bC$ is a given function, and the wave number $k\in (0,2\sqrt{2})$ is real, cf. Figure \ref{Fig1}.

For convenience, we prefer to work in a simplified coordinate system $\bZ^2$ used in \eqref{eq:Vmain} and, therefore,
for describing the content in terms of integer coordinates we tacitly use all the definitions introduced above. Notice that
our problem can be written as follows:
\begin{subequations}
\begin{align}
(\Delta_d +k^2)u(x) &= 0,\quad\quad\  \ \textup{in}\  \mathring{\Omega}, \label{eq:Helmholtz} \\
u(x)&= f(x),\quad \textup{on}\  \partial\Omega. \label{eq:H3}
\end{align}
\end{subequations}
Here, $u(x)=U(w)$ with $x=(x_1,x_2)\in \bZ^2$, and $w = T(x_1,x_2) \in \mathcal{W}$. The discrete Laplacian is given by the following expression
\begin{equation}
\begin{aligned}
\Delta_d u(x)=&u(x+e_1)+u(x-e_1)+u(x+e_2)+u(x-e_2)\\
&+u(x+e_1-e_2)+u(x-e_1+e_2)-6u(x),
\end{aligned}
\end{equation}
where $e_1=(1,0)$, $e_2=(0,1)$ stand for the standard base of $\bZ^2$, cf. Figure \ref{Fig2}.

\begin{figure}[htbp]
\centering
\includegraphics[scale=0.7]{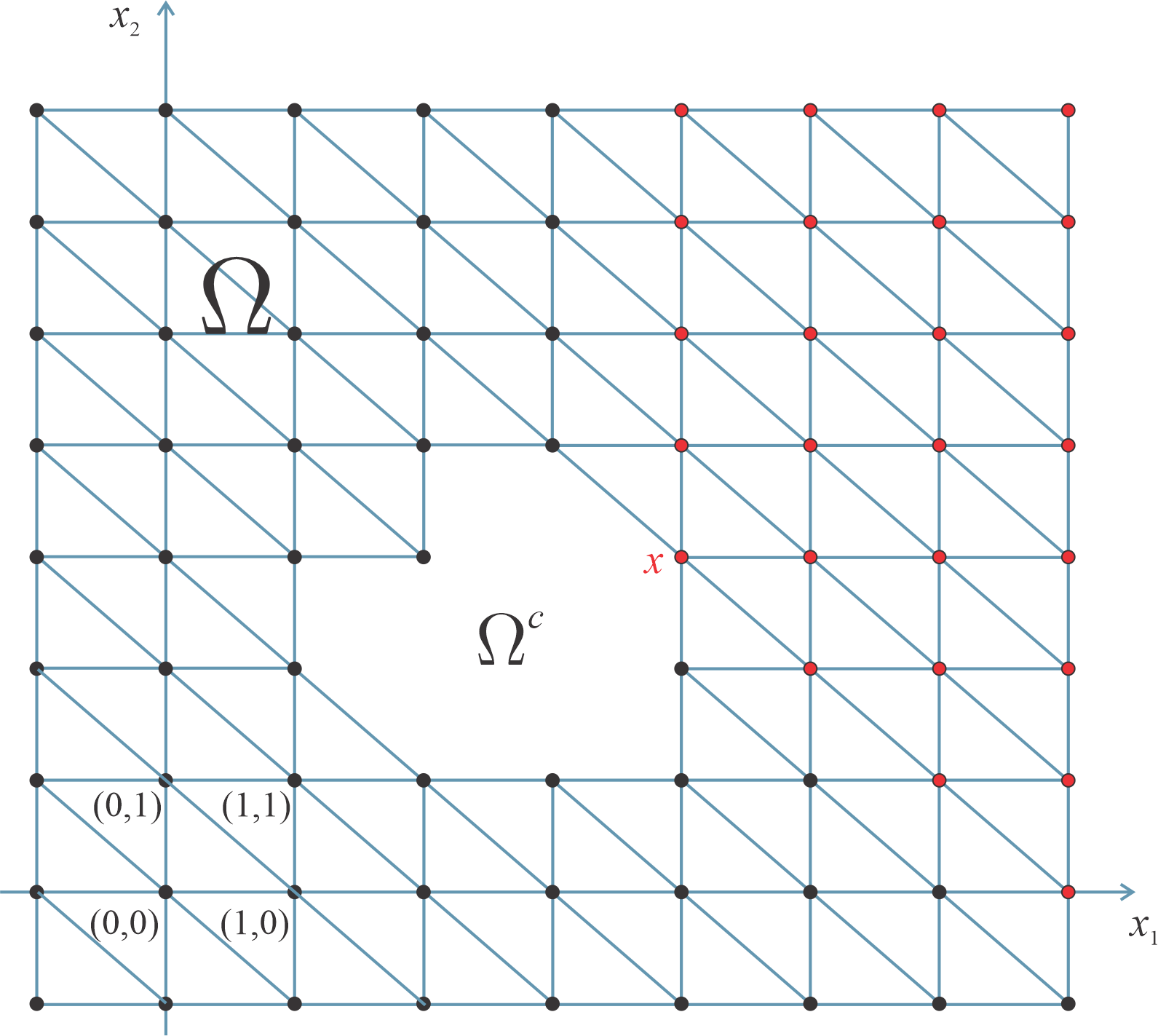}
\caption{Exterior problem in $\bZ^2$. The subset of the ``cone'' $C_{0}(w)$ for the boundary point $w=T(x)$ is represented by the red dots.}
\label{Fig2}
\end{figure}

Thus, we are interested in studying the problem of the existence and uniqueness of the function $u:\Omega\to \bC$ such that $u(x)$ satisfies the discrete Helmoltz equation \eqref{eq:Helmholtz} with $k\in (0,2\sqrt{2})$ and the boundary condition \eqref{eq:H3}. From now on we will refer to this problem as Problem $\mathcal{P}_{\mathrm{ext}}$.

\section{Green's representation formula}\label{sec:3}

Denote by $\cG(x;y)$ the Green's function for the discrete Helmholtz equation \eqref{eq:Helmholtz} centered at $y$ and evaluated at $x$. Then, the function $\cG(x;y)$ satisfies the equation
\begin{equation}\label{eq:greens}
(\Delta_d+k^2)\cG(x;y)=\delta_{x,y},
\end{equation}
where $\delta_{x,y}$ is the Kronecker delta. For brevity, we use the notation $\cG(x)$ for $\cG(x;0)$. Notice that
$\cG(x;y)=\cG(x-y)$.

Using the discrete Fourier transform and the inverse Fourier transform
we get
\begin{equation}\label{eq:gmn}
\cG(x)=\frac{1}{4\pi^2}\int_{-\pi}^{\pi}\int_{-\pi}^{\pi}\frac{e^{\iota(x\cdot \xi)}}{\sigma(\xi;k)}d\xi, \quad \xi=(\xi_1,\xi_2),
\end{equation}
where
\begin{equation}
\begin{aligned}
\sigma(\xi;k^2)&=e^{\iota\xi_1}+e^{-\iota\xi_1}+e^{\iota\xi_2}+e^{-\iota\xi_2}+e^{\iota\xi_1}e^{-\iota\xi_2}+e^{-\iota\xi_1}e^{\iota\xi_2}-6+k^2\\
&=k^2-6+2\cos\xi_1+2\cos\xi_2+2\cos(\xi_1-\xi_2).
\end{aligned}
\end{equation}

The lattice Green's function $\mathcal{G}$ is quite well known when $k^2\in \bC\backslash[0, 9]$ (cf., e.g., \cite{Ho}).
Notice that if $k^2\in \bC\backslash[0, 9]$ then $\sigma\neq 0$ and, consequently, $\mathcal{G}$ in \eqref{eq:gmn} is well defined. In this case
$\mathcal{G}(x)$ decays exponentially when $| x |\to\infty$.

For $k\in (0, 2\sqrt{2})$ we define the lattice Green's function as a pointwise limit of
\begin{equation}
(R_{\lambda+\iota\varepsilon}\delta_{x,0})(x):=\frac{1}{4\pi^2}\int_{-\pi}^{\pi}\int_{-\pi}^{\pi}\frac{e^{\iota x\cdot\xi }d\xi_1 d\xi_2}{\sigma(\xi;k^2+\iota\varepsilon)}
\end{equation}
as $k^2+\iota\varepsilon \to k^2+\iota 0$ and denote it again by $\mathcal{G}(x)$, i.e., $\mathcal{G}(x)=(R_{\lambda+\iota 0}\delta_{x,0})(x)$, cf. \cite{DK1}. Notice that $\mathcal{G}(x)$ is a solution to equation \eqref{eq:greens} and satisfies equalities
\begin{equation}\label{eq:lateq}
\mathcal{G}(x_1,x_2)=\mathcal{G}(x_2,x_1)=\mathcal{G}(-x_1,-x_2)=\mathcal{G}(x_1+x_2,-x_2)
\end{equation}
for all $x=(x_1,x_2) \in \bZ^2$.

In order to simplify further arguments, let us introduce the following vectors:
\begin{equation*}
\begin{aligned}
e_3=e_1-e_2,\quad e_4=-e_1,\quad e_5=-e_2,\quad e_6=-e_3.
\end{aligned}
\end{equation*}
Consider a region $R$ in $\bZ^2$. Denote by $(\partial R)_j$, $j=1,...,6$, a set of all boundary points $y\in\partial R$ such that $y-e_j\in \Rint$ and call it the sides of the boundary $R$. Clearly, $\partial R$ is the union of its six sides: $\partial R=\cup_{j=1}^6 (\partial R)_j$. Notice that a boundary point $y$ can simultaneously belong to all six sides of $R$. However, in our arguments presented below it will be always clear which side is needed to be considered. Under this condition, we define the discrete derivative in the outward normal direction $e_j$, $j=1,\dots,6$,
\begin{equation}
\mathcal{T}u(y)=u(y)-u(y-e_j), \quad y\in(\partial R)_j.
\end{equation}

Let us introduce the following set $H_0=\{(0,0)\}$ and then define $H_{N}$, $N\in\mathbb{N}$, with the help of recurrence formula
\begin{equation}
H_{N}:=\bigcup_{x\in H_{N-1}}F_{x}
\end{equation}
with $\mathring{H}_N:=H_{N-1}$, and $(\partial H)_{N}:=H_{N}\backslash \mathring{H}_N$.

Recall a Green's representation formula for a finite region, cf. \cite{DK1}.

\begin{theorem}\label{theo:repr} Let $R$ be a finite region. Then, for a given function $u:R\to\bC$ and any point $x\in \Rint$, we have a discrete Green's representation formula
\[
u(x)=\sum_{y\in\partial R}\big(u(y)\mathcal{T}\cG(x-y)-\cG(x-y)\mathcal{T}u(y)\big)+\sum_{y\in \Rint}\cG(x-y)(\Delta_d+k^2) u(y).
\]
In particular, if $u$ is a solution to the discrete Helmholtz equation
\[
(\Delta_d+k^2)u(x)=0\quad \textup{in}\  \Rint,
\]
then
\begin{equation}\label{eq:repr}
u(x)=\sum_{y\in\partial R}\big(u(y)\mathcal{T}\cG(x-y)-\cG(x-y)\mathcal{T}u(y)\big).
\end{equation}
\end{theorem}

Finally, for a finite region $R$, recall a discrete analogue of Green's second identity
\begin{equation}\label{eq:GRsecond}
\sum_{x\in \Rint}(u(x)\Delta_d v(x)-v(x)\Delta_d u(x))=\sum_{y\in \partial R} (u(y)\mathcal{T}v(y)-v(y)\mathcal{T}u(y)).
\end{equation}

Now let us give a definition of a radiating solution on $\Omega$ when $k^2\in (0,8)$.
We say that $u:\Omega \to\bC$ satisfies the radiation condition at infinity if
\begin{equation}
\label{eq:radcond}
\left\{\begin{aligned} u(x)&=O(| x|^{-\frac12}),\\
u(x+e_j)&=e^{\iota \xi^*_j(\alpha,k)} u(x)+O(| x|^{-\frac32}),\quad j=1,2,
\end{aligned}
\right.
\end{equation}
with the remaining term decaying uniformly in  all directions $x/| x|$, where $x$ is characterized as $x_1=| x|\cos\alpha$,  $x_2=| x|\sin\alpha$, $0\le \alpha < 2\pi$. Here, $\xi^*_j(\alpha,k)$ is the $j$th coordinate of the point $\xi^*(\alpha, k)$.
Recall that the point $\xi^*(\alpha, k)=\xi^*=(\xi^*_1, \xi^*_2)$ is a unique solution to the following system of equations
\begin{eqnarray*}
2\zeta(\sin\xi_1+\sin(\xi_1-\xi_2))&=&\cos\alpha,\\
2\zeta(\sin\xi_2-\sin(\xi_1-\xi_2))&=&\sin\alpha,\\
k^2-6+2\cos\xi_1+2\cos\xi_2+2\cos(\xi_1-\xi_2)&=&0,
\end{eqnarray*}
where $\zeta$ is a positive constant, cf. \cite{DK1}.

\begin{lemma}
\label{lemma:zeta}
Let $k^2\in (0,8)$, and the function $u$ satisfies the radiation condition at infinity \eqref{eq:radcond}. Then, for any boundary point $y\in(\partial H_N)$, we have
\begin{equation}
\mathcal{T}u(y)=\zeta(y,k) u(y)+O(| y|^{-\frac32}) \quad \textup{as}\ N\to \infty
\end{equation}
such that $\mathrm{\Im m}\,\zeta(y,k)>0$.
\end{lemma}
\begin{proof} First, let us consider the case $y\in(\partial H_N)_1$. Then from \eqref{eq:radcond}, we get
\[
\mathcal{T}(y)=u(y)-u(y-e_1)=\zeta(y,k) u(y) + O(| y|^{-\frac32}),
\]
where $\zeta(y,k)=1-e^{-\iota \xi^*_1(\alpha,k)}$, and $\alpha \in (-\pi/4, \pi/2)$ for $y\in(\partial H_N)_1$. Thus, $\sin(\xi^*_1)>0$ and, consequently, $\mathrm{\Im m}\,\zeta(y,k)>0$.

For the case $y\in(\partial H_N)_2$ the discrete derivative in the outward normal
direction takes the form $\mathcal{T}u(y)= u(y)-u(y-e_2)$.  From \eqref{eq:radcond} we easily derive
\[
\mathcal{T}(y)=\zeta(y,k) u(y) + O(| y|^{-\frac32}),
\]
where $\zeta(y,k)=1-e^{-\iota \xi^*_2(\alpha,k)}$. In this case $\alpha \in (0, 3\pi/4)$ which implies that $\mathrm{\Im m}\,\zeta(y,k)>0$.

For $y\in(\partial H_N)_3$ we need to consider $\mathcal{T}u(y)=u(y)-u(y-e_3)$. From \eqref{eq:radcond} we have
\[
u(y'+e_1-e_2)=e^{\iota \xi^*_1}u(y'-e_2)+O(| y-e_2|^{-\frac32}),
\]
and
\[
u(y'-e_2)=e^{-\iota \xi^*_2}u(y')+O(| y'|^{-\frac32}),
\]
where $y'=y-e_1+e_2$. Further, we get
\[
u(y'+e_1-e_2)=e^{\iota (\xi^*_1-\xi^*_2)}u(y')+O(| y'|^{-\frac32})
\]
which can be written as
\[
u(y)=e^{\iota (\xi^*_1-\xi^*_2)}u(y-e_3)+O(| y'|^{-\frac32}).
\]
Consequently, we derive
\[
\mathcal{T}(y)=\zeta(y,k) u(y) + O(| y|^{-\frac32}),
\]
where $\zeta(y,k)=1-e^{-\iota (\xi^*_1-\xi^*_2)}$. For $y\in(\partial H_N)_3$ we have $\alpha \in (-\pi/2, 0)$ and, in this case, $0<\xi^*_1-\xi^*_2<\pi$ which again gives us $\mathrm{\Im m}\,\zeta(y,k)>0$.

Similar arguments applied to the remaining sides of the boundary $\partial H_N$ complete the proof.
\end{proof}

For a fixed point $x\in\bZ^2$ and any point $y\in \partial H_{N}$ the radiation conditions \eqref{eq:radcond} implies
\[
\sum_{y\in\partial H_{N}}(u(y)\mathcal{T}\cG(x-y)-\cG(x-y)\mathcal{T}u(y)) \to 0, \quad N\to \infty.
\]
Indeed, for instance, since $\alpha(y-x)$ tends to $\alpha=\alpha(y)$, and, consequently, $\zeta(y-x,k)$ tends to $\zeta(y,k)$ as $| y |\to \infty$ then, for sufficiently large $N$, we have
\[
\begin{aligned}
u(y)\mathcal{T}\cG(x-y)&-\cG(x-y)\mathcal{T}u(y)\\
&=u(y)\cdot O(N^{-\frac32})+\cG(x-y)\cdot O(N^{-\frac32})=O(N^{-2}).
\end{aligned}
\]
Further, Theorem \ref{theo:repr} applied for $\Omega\cap H_{N}$, where $N\in\bN$ is sufficiently large, and then passing to the limit $N\to \infty$, yields the following Green's formula for a radiating solution $u$ to the discrete Helmholtz equation \eqref{eq:Helmholtz}
\begin{equation}\label{eq:repre}
u(x)=\sum_{y\in\partial \Omega}(u(y)\mathcal{T}\cG(x-y)-\cG(x-y)\mathcal{T}u(y)).
\end{equation}

From \eqref{eq:repre},  using results obtained in \cite{DK1}, we can conclude that every radiating solution $u$ to the discrete Helmholtz equation \eqref{eq:H1} has the following asymptotic expansion
\begin{equation}\label{eq:far}
u(x)=-\frac{e^{\iota \mu(\alpha,k)|x|}}{|x|^\frac12}\left\{u_{\infty}(\hat{x})+O\left(\frac{1}{|x|}\right)\right\}, \quad |x|\to \infty,
\end{equation}
where $\mu(\alpha,k):=\xi^*(\alpha,k)\cdot \hat{x}$, $\hat{x}:=x/|x|$, and $\xi^*(\alpha,k)=(\xi^*_1(\alpha,k),\xi^*_2(\alpha,k))$. Here, the function $u_{\infty}(\hat{x})$, known as the far field pattern of $u$, can be expressed  with the help of formula (12) from \cite{DK1}.

Now we can formulate the following statement.
\begin{theorem}\label{th:uni} The Problem $\mathcal{P}_{\mathrm{ext}}$ has at most one radiating solution.
\end{theorem}

\begin{proof}
It is sufficient to show that corresponding homogeneous problem has only the trivial solution.

A discrete analogue of the Green's first identity applied in $\Omega \cap H_N$ has a form (see \cite{DK1})
\begin{equation}\label{eq:GRfirst}
\sum_{x\in \Omega \cap H_N}(\nabla^+_d u(x)\cdot \nabla^+_d v(x)+ \nabla^-_d u(x)\cdot \nabla^-_d v(x)+u(x)\Delta_d v(x))=
\sum_{y\in \partial H_N } u(y)\mathcal{T}v(y).
\end{equation}
Taking $v:=\overline{u}$, we get
\begin{equation}\label{eq:SN}
\sum_{x\in \Omega\cap H_N }\left(|\nabla^+_d u(x)|^2+ |\nabla^-_d u(x)|^2-k^2| u(x)|^2\right)=\sum_{y\in \partial H_N } \overline{u}(y)\mathcal{T}u(y).
\end{equation}

Using Lemma \ref{lemma:zeta}, we rewrite equality \eqref{eq:SN} as follows
\[
\sum_{x\in \Omega\cap H_N}\left(|\nabla^+_d u(x)|^2+ |\nabla^-_d u(x)|^2-k^2| u(x)|^2\right)=\sum_{y\in \partial H_N} \zeta(y)| u(y)|^2+O(N^{-1}).
\]
Taking the imaginary part of the last identity and passing to the limit as $N\to \infty$, we get
\[
\sum_{y\in \partial H_N} | u(y)|^2 \to 0\quad \textup{as}\ N\to \infty.
\]
Further, due to Rellich type theorem \cite{AIM16, IM, DK}, we get $u\equiv 0$ outside of $H_N$ for sufficiently large $N$. Since $\Omega$ satisfies the cone condition then it has the unique continuation property \cite[Theorem 5.7]{AIM16} which implies $u\equiv 0$ in $\Omega$.
\end{proof}

\section{Difference potentials and existence of solution}\label{sec:4}

For any function $\varphi : \partial R \to \bC$ we define difference single layer and double layer potentials as follows
\begin{equation}\label{eq:V}
V\varphi(x)=\sum_{y\in\partial R}\cG(x-y)\varphi(y), \quad \textup{for all}\  x\in \bZ^2,
\end{equation}
and
\begin{equation}\label{eq:W}
W\varphi(x)=\sum_{y\in\partial R}\big(\mathcal{T}\cG(x-y)+\delta_{x,y}\big)\varphi(y), \quad  \textup{for all}\  x\in \bZ^2,
\end{equation}
respectively. Since $\delta_{x,y}=0$ for every $x\in \Rint$ and $y\in \partial R$ then \eqref{eq:repre} can be written as
\[
u(x)=Wu(x)-V(\mathcal{T}u)(x), \quad x\in\mathring{\Omega}.
\]
The role of the summand $\delta_{x,y}$ is clarified by the following result.

\begin{lemma}\label{lem:VW}
For every $x \in \Rint$ we have
$$
(\Delta_d+k^2)V\varphi(x)=0, \quad \textup{and} \quad (\Delta_d+k^2)W\varphi(x)=0.
$$
\end{lemma}
\begin{proof}
For the difference single layer potential we have
\[
(\Delta_d+k^2)V\varphi(x)=\sum_{y\in\partial R}[(\Delta_d+k^2)\cG(x-y)]\varphi(y)=\sum_{y\in\partial R}\delta_{x,y}\varphi(y)=0
\]
for all $x\notin \partial R$.
Similarly we can show the result for the discrete double layer potential when $x\in \Rint$ and $F_x\cap\partial R=\varnothing$. Indeed, we have $(\Delta_d+k^2)\mathcal{T}\cG(x-y)=0$, and $(\Delta_d+k^2)\delta_{x,y}=0$. Thus, it remains to consider the case $x\in \Rint$ and $F_x\cap\partial R\neq\varnothing$.
Let $y\in F_x\cap(\partial R)_1$. Then, $y_1+1=x_1$, $y_2=x_2$, and we have
\begin{align*}
(\Delta_d+k^2)(\mathcal{T}\cG(x_1-y_1&,x_2-y_2)+\delta_{x_1,y_1}\delta_{x_2,y_2})\\
&=(\Delta_d+k^2)(\cG(x_1,x_2;y_1,y_2)-\cG(x_1,x_2;y_1+1,y_2))\\
&\ +(\Delta_d+k^2)\delta_{x_1,y_1}\delta_{x_2,y_2}\\
&=\delta_{x_1,y_1}\delta_{x_2,y_2}-\delta_{x_1,y_1+1}\delta_{x_2,y_2}+ \delta_{x_1-1,y_1}\delta_{x_2,y_2}\\
&=0-1+1=0.
\end{align*}
Arguing analogously for the other sides of the boundary $\partial R$, we finally obtain
\[
(\Delta_d+k^2)W\varphi(x)=0, \quad x \in \Rint.
\]
\end{proof}
As a consequence of Lemma \ref{lem:VW}, we have
\[
V\varphi(x)=\sum_{y\in\partial \Omega}\cG(x-y)\varphi(y), \quad x\in \mathring{\Omega},
\]
and
\[
W\varphi(x)=\sum_{y\in\partial \Omega}\left(\mathcal{T}\cG(x-y)+\delta_{x,y}\right)\varphi(y), \quad x\in \mathring{\Omega},
\]
are radiating solutions to the equation \eqref{eq:H1} for any function $\varphi:\partial \Omega \to \bC$. From the proof of Lemma \ref{lem:VW}, it also follows that if $\Omega^c\neq\varnothing$ and $\partial\Omega=\partial\Omega^c$ then
\[
W'\varphi(x)=\sum_{y\in\partial \Omega^c}\mathcal{T}\cG(x-y)\varphi(y)
\]
is a radiating solution to the equation \eqref{eq:Helmholtz}.

Case I: $\Omega=\bZ^2$.  If $y_i$ is a point of intersection of several sides, we choose and fix only one side of the boundary in order to reduce number of numerical computations. Let $m$ be a number of points of $\partial \Omega$. Then, $\partial \Omega$ can be represented as a sequence of points $y_1, y_2, ... , y_m$ such that $y_i=y_j$ if and only if $i=j$ for all $1\le i,j\le m$.
Thus, in \eqref{eq:V} we will have only one summand connected with the boundary point $y_i$ and denote by  $\widetilde{V}$ the corresponding difference potential. Further, from the given function $f$ on $\partial\Omega$, we form a vector $F=(f_1,\dots,f_m)^\top$, $f_i:=f(y_i)$. Similarly, for an unknown function $\varphi$, we write $\Phi=(\varphi_1,\dots,\varphi_m)^\top$, $\varphi_i:=\varphi(y_i)$.

We look for a solution to the Problem $\mathcal{P}_{\mathrm{ext}}$ in the form
\begin{equation}\label{eq:sol1}
u(x)=\widetilde{V}\varphi(x)=\sum_{i=1}^m\cG(x-y_i)\varphi_i, \quad x\in \mathring{\Omega}.
\end{equation}
As in the proof of Lemma \ref{lem:VW}, it can be easily shown that $u$ is a radiating solution to the equation \eqref{eq:Helmholtz}, and it only need to satisfy the boundary conditions \eqref{eq:H3}.
Then, \eqref{eq:H3} implies the following linear system of boundary equations
\begin{equation}\label{eq:BS}
\cH\Phi=F,
\end{equation}
where
\[
\cH=\begin{pmatrix}
    \cG(y_1-y_1) & \cG(y_1-y_2) & \cG(y_1-y_3) & \dots & \cG(y_1-y_m) \\
    \cG(y_2-y_1) & \cG(y_2-y_2) & \cG(y_2-y_3) & \dots & \cG(y_2-y_m) \\
    \hdotsfor{5} \\
    \cG(y_m-y_1) & \cG(y_m-y_2) & \cG(y_m-y_3) & \dots & \cG(y_m-y_m)
\end{pmatrix}.
\]

\begin{lemma} The linear system of boundary equations \eqref{eq:BS} is uniquely solvable.
\end{lemma}
\begin{proof} Due to the Rouch\'{e}-Capelli theorem it is sufficient to proof that the homogeneous system
\begin{equation}\label{eq:hom}
\cH\Phi=0
\end{equation}
has only the trivial solution. Let $\Phi^*=(\varphi^*_1,\dots,\varphi^*_m)^\top$ be a solution to \eqref{eq:hom}. Then,
\[
u(x)=\widetilde{V}\varphi^*(x)
\]
is a radiating solution to the homogeneous Problem $\mathcal{P}_{\mathrm{ext}}$. Therefore, due to Theorem \ref{th:uni}, we have $u\equiv 0$ in $\Omega$. Since $\Omega=\bZ^2$ then at any boundary point $y_i\in\partial\Omega$ we have
\[
0=(\Delta_d+k^2)u(y_i)=(\Delta_d+k^2)\widetilde{V}\varphi^*(y_i)=\sum_{i=1}^m\delta_{y_i,y}\varphi^*(y)=\varphi^*(y_i)=\varphi^*_i
\]
for all $1\le i\le m$.
\end{proof}

Case II: $\mathring{\Omega}\cup \partial \Omega \cup \mathring{\Omega}^c=\bZ^2$. By our assumption on $\partial\Omega^c$ (cf., Section \ref{sec:2}) we can represent it as a sequence of points $y_1, y_2, ... , y_m$ such that $y_i=y_j$ if and only if $i=j$ for all $1\le i,j\le m$.
Denote by $n_i$ the number of the sides of the boundary $\partial\Omega^c$ that $y_i$ belongs to. Then, using the same notation $F=(f_1,\dots,f_m)^\top$, $f_i:=f(y_i)$, and $\Phi=(\varphi_1,\dots,\varphi_m)^\top$, $\varphi_i:=\varphi(y_i)$, we look for a solution to the Problem $\mathcal{P}_{\mathrm{ext}}$ in the form
\begin{equation}\label{eq:sol2}
\begin{aligned}
u(x)&=W'\varphi(x)+\iota\eta V\varphi(x)\\
&=\sum_{j=1}^m \sum_{l=1}^{n_j}\big(\cG(x-y_j)-\cG(x-y^-_{jl})\big)\varphi_j+
\iota\eta\sum_{j=1}^m n_j\cG(x-y_j)\varphi_j\\
&=(1+\iota\eta)\sum_{j=1}^m n_j\cG(x-y_j)\varphi_j-\sum_{j=1}^m \sum_{l=1}^{n_j}\cG(x-y^-_{jl})\varphi_j ,
\end{aligned}
\end{equation}
where $\eta\neq 0$  is a real coupling parameter, and $y^-_{jl}$ , $l=1,..., n_j$, are interior points of $\Omega^c$ which we encounter in the expression $\mathcal{T}\cG(x-y)$ that is the discrete derivative in the outward normal direction with respect to $\Omega^c$. Due to Lemma \ref{lem:VW}, $u$ is a radiating solution to the equation \eqref{eq:Helmholtz}, and we only need to satisfy the boundary conditions \eqref{eq:H3}. From \eqref{eq:H3} we get the following linear system of boundary equations
\begin{equation}\label{eq:BS2}
(1+\iota\eta)\cH\mathcal{N}\Phi - \cK\Phi=F,
\end{equation}
where $\cH$ is defined as above, $\mathcal{N}=\mathrm{diag}(n_1,n_2,\dots, n_m)$ is diagonal matrix and
\[
\cK=\begin{pmatrix}
    \sum_{l=1}^{n_1}\cG(y_1-y^-_{1l}) & \sum_{l=1}^{n_2}\cG(y_1-y^-_{2l}) &  \dots & \sum_{l=1}^{n_m}\cG(y_1-y^-_{ml}) \\
    \sum_{l=1}^{n_1}\cG(y_2-y^-_{1l}) & \sum_{l=1}^{n_2}\cG(y_2-y^-_{2l}) &  \dots & \sum_{l=1}^{n_m}\cG(y_2-y^-_{ml}) \\
    \hdotsfor{4} \\
    \sum_{l=1}^{n_1}\cG(y_m-y^-_{1l}) & \sum_{l=1}^{n_2}\cG(y_m-y^-_{2l}) &  \dots & \sum_{l=1}^{n_m}\cG(y_m-y^-_{ml})
\end{pmatrix}.
\]

\begin{lemma} The linear system of boundary equations \eqref{eq:BS2} is uniquely solvable.
\end{lemma}
\begin{proof}
As above it is sufficient to proof that the homogeneous system
\begin{equation}\label{eq:hom2}
(1+\iota\eta)\cH\mathcal{N}\Phi - \cK\Phi=0
\end{equation}
has only the trivial solution. Let $\Phi^*=(\varphi^*_1,\dots,\varphi^*_m)^\top$ be a solution to \eqref{eq:hom2}. Then, a function
\[
u_+(x)=W'\varphi^*(x)+\iota\eta V\varphi^*(x)
\]
is a radiating solution to the homogeneous Problem $\mathcal{P}_{\mathrm{ext}}$. Therefore, due to Theorem \ref{th:uni}, we have $u^+\equiv 0$ in $\Omega$. In particular, $u^+(y_i)=0$ for all $i=1,\dots,m$.
Further, a function
\[
u_-(x)=W\varphi^*(x)+\iota\eta V\varphi^*(x)
\]
satisfies the discrete Helmholtz equation in $\Omega^c$ (cf., Lemma \ref{lem:VW}).
Notice that for any $y_i\in \partial\Omega^c=\partial\Omega$, we have
\[
u_-(y_i)=u_+(y_i)+\sum_{j=1}^m \sum_{l=1}^{n_i}\delta_{y_j,y_i}\varphi_j^*=u_+(y_i)+n_i\varphi_i^*=n_i\varphi_i^*.
\]
Since $(\Delta_d+k^2)u_+(y_i)=(1+\iota\eta)n_i\varphi_i^*$ and $u_+\equiv 0$ in $\Omega$ we get
\[
\sum_{l=1}^{n_i}\mathcal{T}u_+(y_i)=0-(1+\iota\eta)n_i\varphi_i^*=-(1+\iota\eta)n_i\varphi_i^*.
\]
Noting that $u_+(x)=u_-(x)$ in $\mathring{\Omega}^c$, we obtain
\[
\sum_{l=1}^{n_i}\mathcal{T}u^-(y_i)=n_i\varphi_i^*-(1+\iota\eta)n_i\varphi_i^*=-\iota\eta n_i\varphi_i^*.
\]
Hence, using \eqref{eq:GRfirst} for $u_-$ and its complex conjugation $\bar{u}_-$, we get
\[
\begin{aligned}
\sum_{x\in \mathring{\Omega}^c}\big(|\nabla^+_d u_-(x)|^2&+ |\nabla^-_d u_-(x)|^2 - k^2|u_-(x)|^2\big)\\
&=\sum_{y\in \partial \Omega^c} \bar{u}_-(y)\mathcal{T}u_-(y)=-\iota\eta\sum_{i=1}^m n^2_i|\varphi_i^*|^2.
\end{aligned}
\]
Taking the imaginary part of the last equation, we obtain that $\varphi_i^*=0$ for all $i=1,\dots, m$, so $\Phi^*=0$.
\end{proof}

Due to a direct combination of the results obtained above now we have the main conclusions of the present work.

\begin{theorem}\label{theo:main} The Problem $\mathcal{P}_{\mathrm{ext}}$ has a unique radiating solution which is represented as \eqref{eq:sol1} if $\Omega=\bZ^2$ or as \eqref{eq:sol2} if $\mathring{\Omega}\cup\partial\Omega\cup \mathring{\Omega}^c=\bZ^2$, where $\Phi$ is a unique solution to the system of linear equations \eqref{eq:BS} or \eqref{eq:BS2}, respectively.
\end{theorem}

\section{Numerical results}\label{sec:5}

The main difficulty for numerical evaluation of solutions to \eqref{eq:BS} and \eqref{eq:BS2} is to compute the lattice Green's function. For this purpose we apply  the method developed in \cite{BC}. Using 8-fold symmetry, we need only to compute the lattice Green's function $\cG(i,j)$ with $i\ge j \ge 0$. Following to \cite{BC}, let us introduce the vectors
$\cV_{2p}=(\cG(2p,0),\cG(2p-1,1),\dots,\cG(p,p))^\top$ and $\cV_{2p+1}=(\cG(2p+1,0),\cG(2p,1),\dots,\cG(p+1,p))^\top$ that collect all distinct Green’s functions $\cG(i,j)$ with ``Manhattan distances" $\mid i \mid + \mid j \mid$ of $2p$ and $2p+1$, respectively. For any Manhattan distance larger than 1, equation
\begin{equation}
(\Delta_d+k^2)\cG(x)=\delta_{x,0}
\end{equation}
can be written in the matrix form $\gamma_n(k)\cV_n=\alpha_n(k)\cV_{n-1}+\beta_n(k)\cV_{n+1}$
where $\alpha_n(k)$, $\beta_n(k)$ and $\gamma_n(k)$ are sparse matrices (cf., Appendix A). Notice that only the
dimensions of these matrices depend on $n$. It is shown in \cite{BC} that, for any $n\ge 1$, we have
\begin{equation}
\cV_n=A_n(k)\cV_{n-1},
\end{equation}
where the matrices $A_n(k)$ are defined by the following recurrence formula
\begin{equation}
A_n(k)=[\gamma_n(k)-\beta_n(k)A_{n+1}]^{-1}\alpha_n(k).
\end{equation}
They can be computed starting from a sufficiently large $N$ with $A_{N+1}(k)=0$. Here, it is worth mentioning that for $k=2$ we need to choose a better ``initial guess" than $A_{N+1}(k)=0$, since in this case $\mathrm{det}\gamma_n(k)=0$, and the matrix $\gamma_n(k)-\beta_n(k)A_{n+1}$ is not invertible.

Once $A_n(k)$ are known, we have $\cV_n=A_n(k)\dots A_1(k)\cV_{0}$, where $\cV_{0}=\cG(0,0)$. In particular, $\cV_1=\cG(1,0)=A_1(k)\cG(0,0)$ which, together with $6\cG(1,0)-(6-k^2)\cG(0,0)=1$, gives $\cG(0,0)=1/[6A_1(k)-6+k^2]$. This completes the calculation of the Green's function using elementary operations and no integrals. Notice also one more important advantage of this method. The $A_n(k)$ matrices are calculated coming down from asymptotically large Manhattan distances.  As they are propagated towards smaller Manhattan distances, it definitely gives us the physical solution.

Finally, we demonstrate our theoretical and numerical approaches on the following Problem $\mathcal{P}_{\mathrm{ext}}$:
let $\mathring{\Omega}^c$ and $\partial \Omega^c$ be the sets of the following points $(2,2)$, $(3,2)$, $(3,3)$ and
$(2,1)$, $(3,1)$, $(4,1)$, $(4,2)$, $(4,3)$, $(3,4)$, $(2,4)$, $(2,3)$, $(1,3)$, $(1,2)$, respectively. Further, we set $\Omega=\bZ^2\backslash \mathring{\Omega}^c$ with $\partial \Omega= \partial \Omega^c$, cf. Figure \ref{Fig2}. Clearly $\Omega$ satisfies the cone condition, and we have a decomposition $\mathring{\Omega}\cup \partial \Omega \cup \mathring{\Omega}^c=\bZ^2$. For this example, we take $f(y)\equiv 1$ on $\partial \Omega$ and $k=\sqrt{2}$. Due to Theorem \ref{theo:main}, the problem is uniquely solvable, and the solution can be found as
\[
u(x)=W'\varphi(x)+\iota V\varphi(x).
\]
Here, we take $\eta=1$. One can try to minimize the condition number of corresponding matrices by the proper choice of $\eta$, but it is not our goal at the moment. The vector $\Phi=(\varphi_1,\dots,\varphi_{10})^\top$ is a unique solution to equation \eqref{eq:BS2}.
In order to solve obtained system of linear equations and then find the solution $u$, we have developed MATLAB code that uses the efficient method described above to compute Green's functions. As a technical aside, these data were obtained in several minutes on a regular desktop. Some results of numerical evaluations are plotted in Figure \ref{Fig3}. Some key features of numerical solutions can be immediately observed. Namely, due to the small hole, we notice some kind of symmetry of $\mathrm{\Re e}\,u$ and $|u|$ in the macro level, however, as the plot on Figure \ref{Fig3} (e) shows, we do not have an exact symmetry at the micro level. Besides, we also see some interference effects on the density plots of  $\mathrm{\Re e}\,u$ and $|u|$.

\begin{figure}[t!]  %[htbp]
\begin{center}
\subfigure[The density plot of $\mathrm{Re}\,u$.]{\includegraphics[scale=0.45]{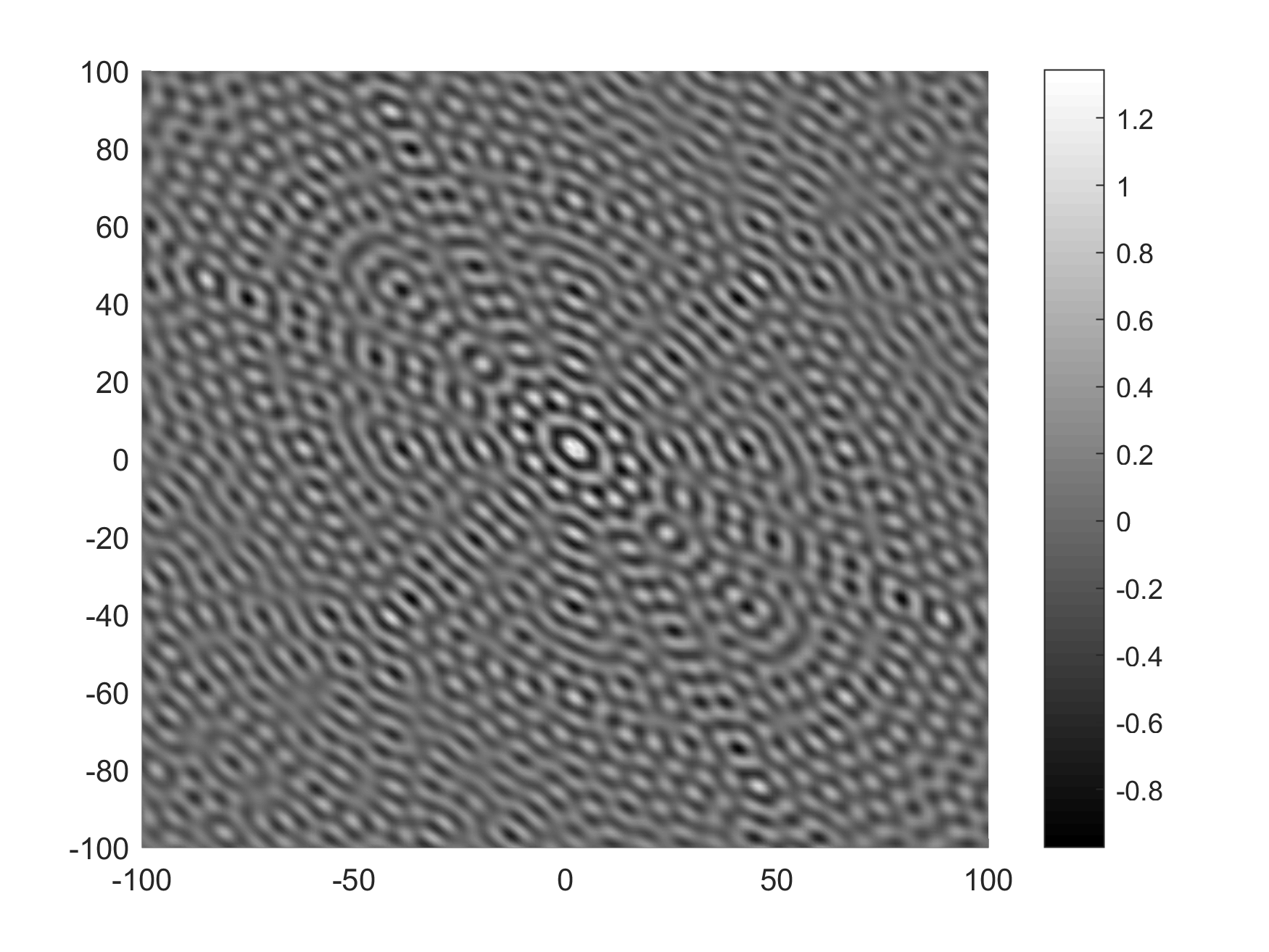}}
\subfigure[The density plot of $|u|$ in $\bZ^2$.]{\includegraphics[scale=0.45]{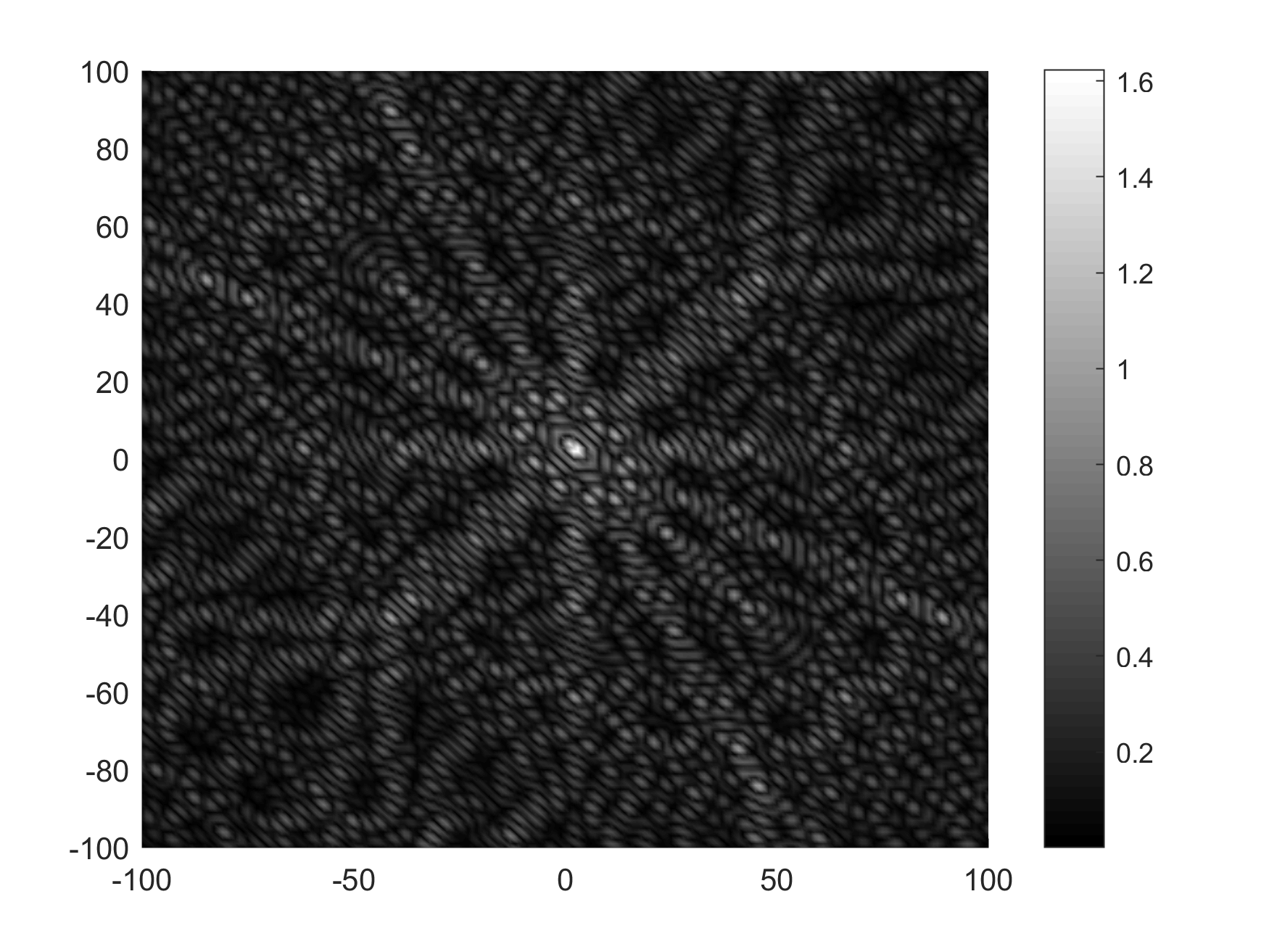}}
\end{center}
\begin{center}
\subfigure[The density plot of $\mathrm{Re}\,u$.]{\includegraphics[scale=0.45]{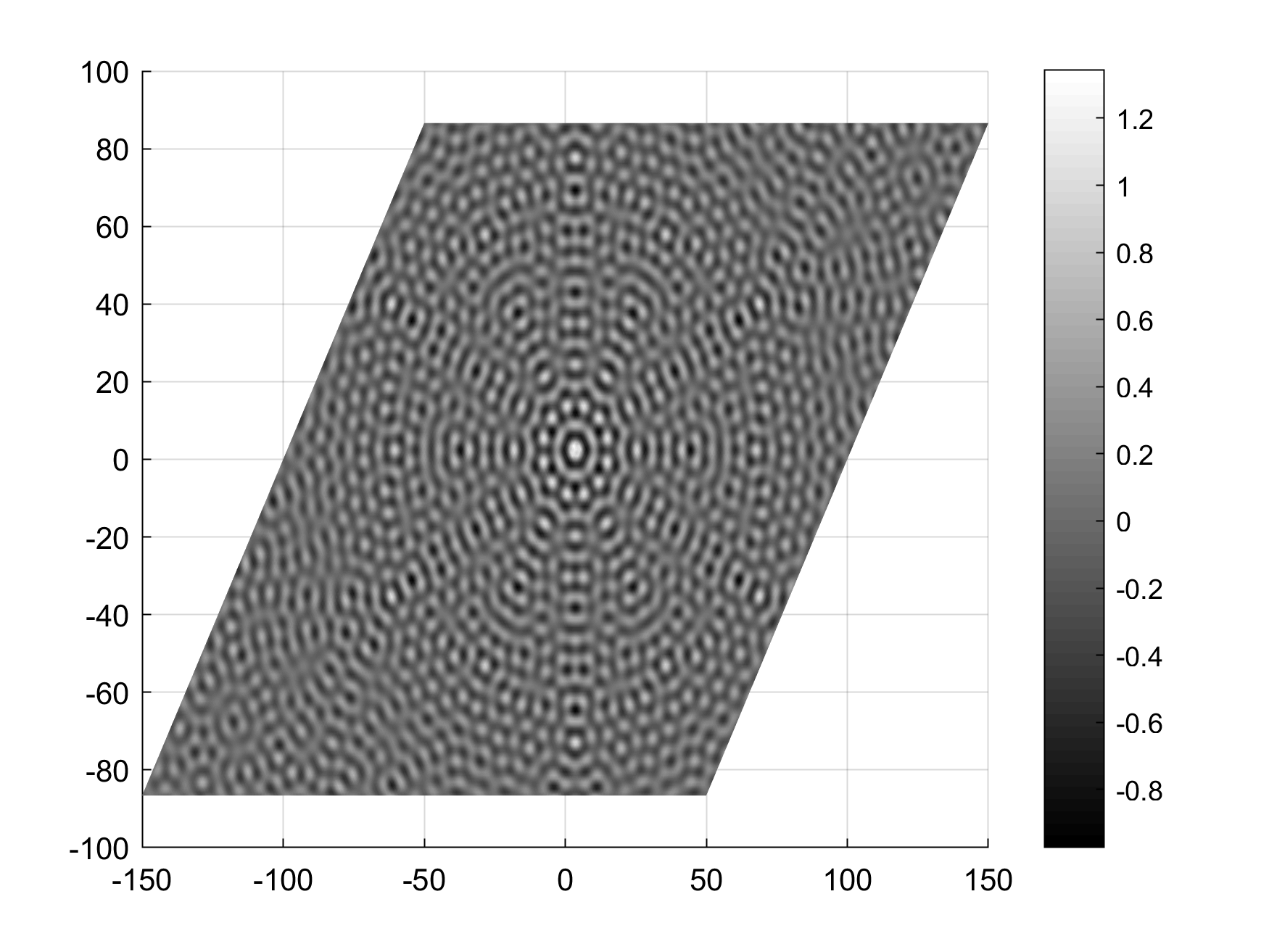}}
\subfigure[The density plot of $|u|$.]{\includegraphics[scale=0.45]{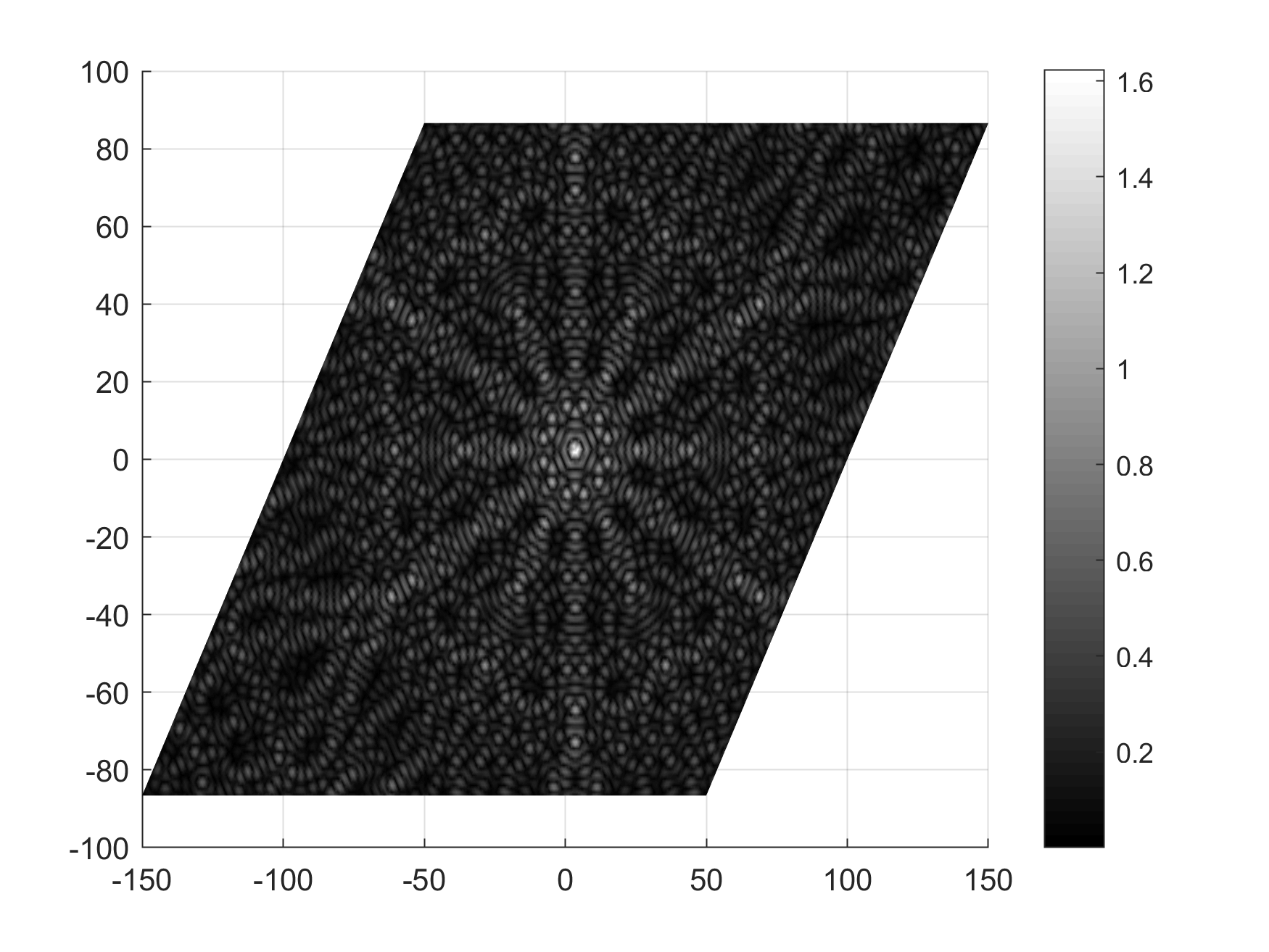}}
\end{center}
\begin{center}
\subfigure[The graph of $\mathrm{Re}\,u(\cdot,2)$ .]{\includegraphics[scale=0.4]{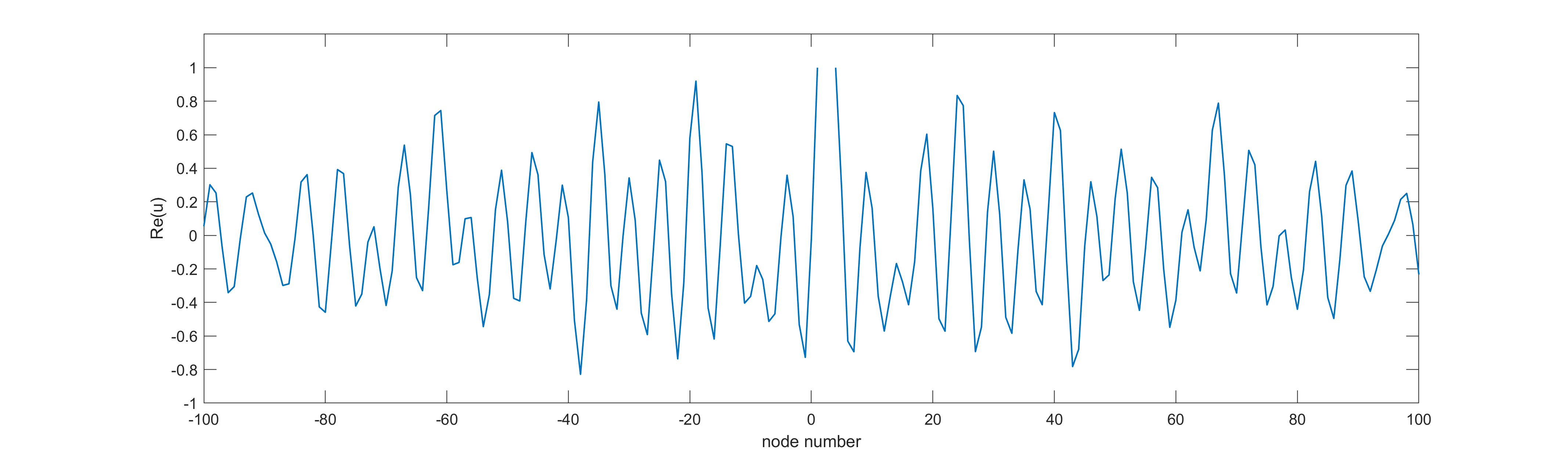}}
\end{center}
\caption{Plots in (a), (b), and (e) are represented in $\bZ^2$, while (c) and (d) are represented in original coordinates of the triangular lattice $\mathfrak{T}$. $k=\sqrt{2}$.}
\label{Fig3}
\end{figure}

\section{Discussion}

In this paper, we have constructed the discrete scattering theory for the two-dimensional Helmholtz equation with the real wave number $k\in (0,2\sqrt{2})$ for the triangular lattice. Similarly to the continuum theory, we used the notion of radiating solution for the continuous Helmholtz equation and solved the exterior Dirichlet problem without passing to the complex wave number. Here, it is worth mentioning that due to the more complex form of the radiation condition for $k\in (2\sqrt{2}, 3)$ than \eqref{eq:radcond} we have some difficulties to prove the uniqueness result for this case using the presented approach. Clearly, we can introduce a small dumping parameter, pass it to the complex wave number, and then use the limiting absorption principle to get the desired solution, but this is not our goal and, therefore, we have restricted ourselves only to the case $k\in (0,2\sqrt{2})$.

Finally, depending on the objectives of the investigation one can consider different spaces on lattices, but for our purposes it is sufficient to take the space $\ell^\infty_R(\Omega)$, which is a Banach space of all bounded sequences on $\Omega\subset \bZ^2$ that satisfy the radiation condition \eqref{eq:radcond}.

\section{Appendix A : Sparse matrices}\label{secA1}

The sparse matrices $\alpha_n(k)$, $\beta_n(k)$ and $\gamma_n(k)$ are defined as follows:
if $n=2p$  then  $\alpha_{2p}(k)$ is a $(p+1)\times p$ matrix such that
$\alpha_{2p}(k)\mid_{i,i}=1$, $i=\overline{1,p}$,
$\alpha_{2p}(k)\mid_{i,i-1}=1$, $i=\overline{2,p}$,
while $\alpha_{2p}(k)\mid_{p+1,p}=2$,
and all other matrix elements are zero. The $\beta_{2p}(k)$ is a $(p+1)\times (p+1)$ matrix
such that $\beta_{2p}(k)\mid_{i,i}=1$, $i=\overline{1,p}$,
$\beta_{2p}(k)\mid_{i,i+1}=1$, $i=\overline{2,p}$,
while $\beta_{2p}(k)\mid_{p+1,p+1}=\beta_{2p}(k)\mid_{1,2}=2$, and all other matrix elements are zero.
The $\gamma_{2p}(k)$ is a $(p+1)\times (p+1)$ matrix such that $\gamma_{2p}(k)\mid_{i,i}=6-k^2$, $i=\overline{1,p+1}$,
$\gamma_{2p}(k)\mid_{i,i+1}=\gamma_{2p}(k)\mid_{i,i-1}=-1$, $i=\overline{2,p}$, and $\gamma_{2p}(k)\mid_{1,2}=\gamma_{2p}(k)\mid_{p+1,p}=-2$.

If $n=2p+1$ then $\alpha_{2p+1}(k)$ is a $(p+1)\times (p+1)$ matrix such that
$\alpha_{2p+1}(k)\mid_{i,i}=1$, $i=\overline{1,p+1}$,
$\alpha_{2p+1}(k)\mid_{i,i-1}=1$, $i=\overline{2,p+1}$,
and all other matrix elements are zero. The $\beta_{2p+1}(k)$ is a
$(p+1)\times (p+2)$ matrix such that $\beta_{2p+1}(k)\mid_{i,i}=1$, $i=\overline{1,p+1}$,
$\beta_{2p+1}(k)\mid_{i,i+1}=1$, $i=\overline{2,p+1}$,
while $\beta_{2p+1}(k)\mid_{1,2}=2$, and all other matrix elements are zero.
The $\gamma_{2p+1}(k)$ is a $(p+1)\times (p+1)$ matrix such that $\gamma_{2p+1}(k)\mid_{i,i}=6-k^2$, $i=\overline{1,p}$,
$\gamma_{2p+1}(k)\mid_{p+1,p+1}=5-k^2$, while $\gamma_{2p+1}(k)\mid_{i,i+1}=-1$, $i=\overline{2,p}$, $\gamma_{2p+1}(k)\mid_{i,i-1}=-1$, $i=\overline{2,p+1}$, and $\gamma_{2p+1}(k)\mid_{1,2}=-2$. Finally, $\gamma_1(k)$ is a $1\times 1$ matrixs with an element $4-k^2$.

\section*{Acknowledgments}

This work was supported by Shota Rustaveli National Science Foundation of Georgia (SRNSFG) [FR-21-301]

%\bf{Declaration of competing interest}

%The authors declare that they have no known competing financial interests or personal relationships that could have
%appeared to influence the work reported in this paper.

\end{document}